\documentclass[11pt]{amsart}
\usepackage{amsmath}
\usepackage{amssymb} 
\usepackage{geometry}
\def\beqnn{\begin{eqnarray*}}\def\eeqnn{\end{eqnarray*}}

\newtheorem{theorem}{Theorem}[section]
\newtheorem{lemma}[theorem]{Lemma}
\newtheorem{proposition}[theorem]{Proposition}
\newtheorem{corollary}[theorem]{Corollary}

\theoremstyle{remark}
\newtheorem{remark}[theorem]{Remark}

\theoremstyle{definition}

\theoremstyle{question}
\newtheorem{question}[theorem]{Question}
\theoremstyle{problem}

\theoremstyle{conjecture}

\numberwithin{equation}{section}

\begin{document}

\begin{center}
\title{ {On the operators of Hardy-Littlewood-P\'olya type}}
\end{center}

\author{Jianjun Jin}
\address{School of Mathematics Sciences, Hefei University of Technology, Xuancheng Campus, Xuancheng 242000, P.R.China}
\email{jin@hfut.edu.cn, jinjjhb@163.com}





\begin{abstract}
In this paper we introduce and study several new Hardy-Littlewood-P\'olya-type operators.  In particular, we study a Hardy-Littlewood-P\'olya-type operator induced by a positive Borel measure on $[0,1)$. We establish some sufficient and necessary conditions for the boundedness (compactness) of these operators. We also determine the exact values of the norms of the Hardy-Littlewood-P\'olya-type operators for certain special cases.
\end{abstract}

\maketitle
\pagestyle{plain}

\section{\bf {Introduction and main results}}
Throughout this paper, for two positive numbers $A, B$, we write $A \preceq B$, or $A \succeq B$, if there exists a positive constant $C$  independent of the arguments such that $A \leq C B$, or $A \geq C B$, respectively. We will write $A \asymp B$ if both $A \preceq B$ and $A \succeq B$.

Let $p>1$. We denote the conjugate of $p$ by $p'$, i.e., $\frac{1}{p}+\frac{1}{p'}=1$.  Let $l^p$ be the space of sequences of complex numbers, i.e.,
\begin{equation*}l^{p}:=\{a=\{a_n\}_{n=1}^{\infty}: \|a\|_{p}=(\sum_{n=1}^{\infty} |a_{n}|^p)^{\frac{1}{p}}<+\infty \}.\end{equation*}
For $a=\{a_n\}_{n=1}^{\infty}$, the Hardy-Littlewood-P\'olya operator $\mathbf{H}$ is defined as
$$\mathbf{H}(a)(m):=\sum_{n=1}^{\infty}\frac{a_n}{\max\{m, n\}},  \, m\in \mathbb{N}. $$

It is well known (see \cite[page 254]{HLP}) that 
\begin{theorem}\label{m-0}
Let $p>1$. Then $\mathbf{H}$ is bounded on $l^p$ and the norm of $\mathbf{H}$ is $p+p'$. \end{theorem}

Hardy-Littlewood-P\'olya operator is related to some important topics in analysis and there have been many results about this operator and its analogous and generalizations. The classical results of this topic can be founded in the famous monograph \cite{HLP}.  In the past three decades,  the so-called Hilbert-type operators, including Hardy-Littlewood-P\'olya-type operators,  have been extensively studied by Yang and his coauthors,  see the survey \cite{YR} and Yang's book \cite{Y3}. For more recent results see for example \cite{WHY} and \cite{YZ}.  Fu et al. have studied in \cite{FWL} some $p$-adic Hardy-Littlewood-P\'olya-type operators.Very recently, in the work \cite{B}, Brevig established some norm estimates for certain Hardy-Littlewood-P\'olya-type operators in terms of the Riemann zeta function.  Some further results have been  obtained in \cite{B-1}.  

In this paper, we first introduce and  study the following operator of the Hardy-Littlewood-P\'olya type,
$$\mathbf{H}^{\mu, \nu}_{\alpha, \beta, \gamma}(a)(m):=m^{\frac{1}{p}[(\alpha-1)+\alpha \mu]}\sum_{n=1}^{\infty}
\frac{n^{\frac{1}{p'}[(\beta-1)-(p'-1)\beta \nu]}}{ [\max\{m^{\alpha}, n^{\beta}\}]^{\gamma}}a_n, a=\{a_n\}_{n=1}^{\infty}, m\geq 1, $$
where $\gamma>0$, $0<\alpha, \beta \leq 1$, $-1<\mu, \nu <p-1$.

The operator $\mathbf{H}^{\mu, \nu}_{\alpha, \beta, \gamma}$ reduces to the classical Hardy-Littlewood-P\'olya operator $\mathbf{H}$ when $\alpha=\beta=\gamma=1$, $\mu=\nu=0$. We first study the boundedness of $\mathbf{H}^{\mu, \nu}_{\alpha, \beta, \gamma}$.  We will provide a sufficient and necessary condition for the boundedness of $\mathbf{H}^{\mu, \nu}_{\alpha, \beta, \gamma}$ in terms of the parameters $\gamma, \mu, \nu$ and prove that
\begin{theorem}\label{m-1-1}
Let $p>1$, $\gamma>0$, $0<\alpha, \beta \leq 1$ and $-1<\mu, \nu <p-1$. Let $\mathbf{H}^{\mu, \nu}_{\alpha, \beta, \gamma}$ be defined as obove. Then  $\mathbf{H}^{\mu, \nu}_{\alpha, \beta, \gamma}$ is bounded on $l^p$  if and only if $p(\gamma-1)-(\mu-\nu)\geq 0.$\end{theorem}

When $p(\gamma-1)-(\mu-\nu)= 0$, i.e.,  $\gamma=1+\frac{\mu-\nu}{p}$, we use $\widetilde{\mathbf{H}}^{\mu, \nu}_{\alpha, \beta}$ to denote the operator $\mathbf{H}^{\mu, \nu}_{\alpha, \beta, \gamma}$.  That is to say, 
$$\widetilde{\mathbf{H}}^{\mu, \nu}_{\alpha, \beta}(a)(m):=m^{\frac{1}{p}[(\alpha-1)+\alpha \mu]}\sum_{n=1}^{\infty}
\frac{n^{\frac{1}{p'}[(\beta-1)-(p'-1)\beta \nu]}}{ [\max\{m^{\alpha}, n^{\beta}\}]^{1+\frac{\mu-\nu}{p}}}a_n,a=\{a_n\}_{n=1}^{\infty},m\geq 1.$$

We denote by  $\|\widetilde{\mathbf{H}}^{\mu, \nu}_{\alpha, \beta}\|$ the norm of $\widetilde{\mathbf{H}}^{\mu, \nu}_{\alpha, \beta}$. We will show the following result, which is an extention of Theorem \ref{m-0}.
\begin{theorem}\label{m-1-2}
Let $p>1$, $0<\alpha, \beta \leq 1$ and $-1<\mu, \nu <p-1$. Let $\widetilde{\mathbf{H}}^{\mu, \nu}_{\alpha, \beta}$ be defined as above.  Then  $\widetilde{\mathbf{H}}^{\mu, \nu}_{\alpha, \beta}$ is bounded on $l^p$ and
\begin{equation}\label{norm}\|\widetilde{\mathbf{H}}^{\mu, \nu}_{\alpha, \beta}\|= \frac{p}{{\alpha}^{\frac{1}{p}}{\beta}^{\frac{1}{p'}}}\left(\frac{1}{1+\mu}+\frac{1}{p-1-\nu}\right).\end{equation}  \end{theorem}

When $\alpha=\beta=1$.  From Theorem \ref{m-1-1}, we know that the operator 
$$\mathbf{H}^{\mu, \nu}_{1, 1, \gamma}(a)(m)=\sum_{n=1}^{\infty} \frac{m^{\frac{\mu}{p}}n^{-\frac{\nu}{p}}}{[\max\{m, n\}]^{\gamma}}a_n, a=\{a_n\}_{n=1}^{\infty},m\geq 1, $$
is not bounded on $l^p$ when $\gamma<1+\frac{\mu-\nu}{p}$. On the one hand, we note that
$$\int_{[0, 1)}t^{\max\{m, n\}-1}(1-t)^{\gamma-1}dt=B(\max\{m, n\}, \gamma)=\frac{\Gamma(\max\{m, n\})\Gamma(\gamma)}{\Gamma(\gamma+\max\{m, n\})}, $$
for $\gamma>0$, $m, n\geq 1$.  Here $B(\cdot, \cdot)$ is the Beta function, which is defined as
$$B(u,v):=\int_{0}^{1}t^{u-1}(1-t)^{v-1}\,dt,\: u>0,v>0.$$
The Gamma function $\Gamma(\cdot) $ is defined as
$$\Gamma(x)=\int_{0}^{\infty}e^{-t} t^{x-1}\,dt,\: x>0.$$
It is known that
$$B(u,v)=\frac{\Gamma(u)\Gamma{(v)}}{\Gamma(u+v)}.$$
For more introductions to these special functions, see \cite{AAR}. On the other hand,  we see from 
\begin{equation}\label{g}
\Gamma(x) = \sqrt{2\pi} x^{x-\frac{1}{2}}e^{-x}[1+r(x)],\, |r(x)|\leq e^{\frac{1}{12x}}-1,\, x>0,\end{equation}
that
$$\frac{\Gamma(\max\{m, n\})\Gamma(\gamma)}{\Gamma(\gamma+\max\{m, n\})} \asymp \frac{1}{[\max\{m, n\}]^{\gamma}},\, \gamma>0, m, n\geq 1.$$
Hence, in order to make $\mathbf{H}^{\mu, \nu}_{1, 1, \gamma}$ to be bounded on $l^p$  when $\gamma<1+\frac{\mu-\nu}{p}$, we let $\lambda$ be a positive Borel measure in $[0, 1)$, and consider the following operator
$$\widehat{\mathbf{H}}^{\mu, \nu}_{\gamma, \lambda}(a)(m):=\sum_{n=1}^{\infty} m^{\frac{\mu}{p}}n^{-\frac{\nu}{p}}\mathcal{I}_\lambda[m,n]a_n,  \, a=\{a_n\}_{n=1}^{\infty}, \,  m\geq 1.$$
Where
\begin{equation}\label{mea}\mathcal{I}_\lambda[m, n]=\int_{[0, 1)}t^{\max\{m, n\}-1}(1-t)^{\gamma-1}d\lambda(t), \,m, n\geq 1.\end{equation}

We will characterize measures $\lambda$ such that $\widehat{\mathbf{H}}^{\mu, \nu}_{\gamma, \lambda}$ is bounded (compact) on $l^p$. To state our results, we introduce the notion of generalized Carleson measure on $[0,1)$. Let $s>0$, let  $\lambda$ be a positive Borel measure on $[0,1)$, we say $\lambda$ is an $s$-Carleson measure if there is a constant $C>0$ such that $$\lambda([t, 1))\leq C (1-t)^s$$ holds for all $t\in [0, 1)$. 
Moreover, an $s$-Carleson measure $\lambda$ on $[0, 1)$  is  said to be a vanishing $s$-Carleson measure, if  it satisfies further that
$$\lim_{t\rightarrow 1^{-}}\frac{\lambda([t, 1))}{(1-t)^s}=0.$$

We shall prove the following criterion for the boundedness of $\widehat{\mathbf{H}}^{\mu, \nu}_{\gamma, \lambda}$.
\begin{theorem}\label{m-1-3}
Let $p>1, \gamma>0$ and  $-1<\mu, \nu<p-1$.  Let $\lambda$ be a positive Borel measure on $[0, 1)$ such that $d\rho(t):=(1-t)^{\gamma-1}d\lambda(t)$ is a finite measure on $[0, 1)$, and  $\widehat{\mathbf{H}}^{\mu, \nu}_{\gamma, \lambda}$ be defined as above. Then $\widehat{\mathbf{H}}^{\mu, \nu}_{\gamma, \lambda}$ is bounded on $l^p$ if and only if
$\rho$ is a $[1+\frac{1}{p}(\mu-\nu)]$-Carleson measure on $[0, 1)$.
\end{theorem}

For the compactness of $\widehat{\mathbf{H}}^{\mu, \nu}_{\gamma, \lambda}$, we shall show that
 \begin{theorem}\label{m-1-4}
Let $p>1, \gamma>0$ and  $-1<\mu, \nu<p-1$.  Let $\lambda$ be a positive Borel measure on $[0, 1)$ such that $d\rho(t):=(1-t)^{\gamma-1}d\lambda(t)$ is a finite measure on $[0, 1)$, and  $\widehat{\mathbf{H}}^{\mu, \nu}_{\gamma, \lambda}$ be defined as above. Then $\widehat{\mathbf{H}}^{\mu, \nu}_{\gamma, \lambda}$ is compact on $l^p$ if and only if
$\rho$ is a vanishing $[1+\frac{1}{p}(\mu-\nu)]$-Carleson measure on $[0, 1)$.
\end{theorem}

The paper is organized as follows.  Two lemmas will be given in the next section.  We will first prove Theorem \ref{m-1-2} in Section 3. The proof of Theorem \ref{m-1-1} will be given in Section 4.  We prove Theorem \ref{m-1-3} and \ref{m-1-4} in Section 5. Final remarks will be presented in Section 6.

\section{\bf {Two lemmas}}

We need the following lemmas in the proof of our main results of this paper. 
\begin{lemma}\label{lem-1}
Let $p>1$, $0<\alpha, \beta \leq 1$ and $-1<\mu, \nu <p-1$.  We define
 $$E(m):=\sum_{n=1}^{\infty}
\frac{n^{\beta-1}}{[\max\{m^{\alpha}, n^{\beta}\}]^{1+\frac{\mu-\nu}{p}}} \cdot \frac{m^{\frac{\alpha(1+\mu)}{p}}}{n^{\frac{\beta(1+\nu)}{p}}}, \,  m\geq 1; $$
$$F(n):=\sum_{m=1}^{\infty}
\frac{m^{\alpha-1}}{[\max\{m^{\alpha}, n^{\beta}\}]^{1+\frac{\mu-\nu}{p}}}\cdot \frac{n^{\frac{\beta(p-1-\nu)}{p}}}{m^{\frac{\alpha(p-1-\mu)}{p}}}, \,  n\geq 1. $$
Then we have
\begin{equation}\label{ineq-1}E(m)\leq \frac{p}{\beta}\left(\frac{1}{1+\mu}+\frac{1}{p-1-\nu}\right),\, m\geq 1,\end{equation}
\begin{equation} \label{ineq-2}F(n)\leq \frac{p}{\alpha}\left(\frac{1}{1+\mu}+\frac{1}{p-1-\nu}\right),\, n\geq 1. \end{equation}
\end{lemma}
\begin{proof}
In view of the assumption, we see that, for $ m\geq 1$, 
\begin{equation*}
E(m)\leq \int_{0}^{\infty}
\frac{x^{\beta-1}}{[\max\{m^{\alpha}, x^{\beta}\}]^{1+\frac{\mu-\nu}{p}}} \cdot \frac{m^{\frac{\alpha(1+\mu)}{p}}}{x^{\frac{\beta(1+\nu)}{p}}}\,dx.
\end{equation*}
Consequently,  by the change of variables $s=x^{\beta}$, we obtain that
\begin{eqnarray}
E(m) &\leq&\frac{1}{\beta} \int_{0}^{\infty}
\frac{1}{[\max\{m^{\alpha}, s\}]^{1+\frac{\mu-\nu}{p}}} \cdot \frac{m^{\frac{\alpha(1+\mu)}{p}}}{s^{\frac{1+\nu}{p}}}\,ds \nonumber \\
&=& \frac{1}{\beta}  \int_{0}^{\infty}
\frac{t^{-\frac{1+\nu}{p}}}{[\max\{1, t\}]^{1+\frac{\mu-\nu}{p}}}\,dt\nonumber \\
&=& \frac{1}{\beta}  \int_{0}^{1}t^{-\frac{1+\nu}{p}}dt+\frac{1}{\beta}  \int_{1}^{\infty}t^{-\frac{1+\mu}{p}-1}dt  \nonumber \\
&=& \frac{p}{\beta}\left(\frac{1}{1+\mu}+\frac{1}{p-1-\nu}\right). \nonumber
\end{eqnarray}
This proves (\ref{ineq-1}).  By the similar way, we can obtain that (\ref{ineq-2}) also holds. The lemma is proved.
\end{proof}

\begin{lemma}\label{lem}
Let $\gamma>0, -1<\mu, \nu<p-1$. Let $\lambda$ be a positive Borel measure on $[0, 1)$ and $\mathcal{I}_{\lambda}[m, n]$ be defined as in (\ref{mea}) for $m, n\geq 1$. Set $d\rho(t)=(1-t)^{\gamma-1}d \lambda(t)$.  If $\rho$ is a  $[1+\frac{1}{p}(\mu-\nu)]$-Carleson measure on $[0, 1)$, then
\begin{equation}\label{mun-1}\mathcal{I}_{\lambda}[m, n] \preceq \frac{1}{[\max\{m, n\}]^{1+\frac{1}{p}(\mu-\nu)}}\end{equation}
holds for all $m, n\geq 1$. Furthermore, if $\rho$ is a  vanishing $[1+\frac{1}{p}(\mu-\nu)]$-Carleson measure on $[0, 1)$, then
 \begin{equation}\label{mun-2}
 \mathcal{I}_{\lambda}[m, n] =o\left(\frac{1}{[\max\{m, n\}]^{1+\frac{1}{p}(\mu-\nu)}}\right), \,\,\max\{m, n\}\rightarrow \infty.\end{equation}
\end{lemma}
\begin{proof}
When  $m\geq 1, n\geq 2$, or $m\geq 2, n\geq 1$. We get from integration by parts that
\begin{eqnarray}
\mathcal{I}_{\lambda}[m, n]&=&\int_{0}^1 t^{\max\{m, n\}-1}d\rho(t) \nonumber \\&=&\rho([0,1))-(\max\{m, n\}-1)\int_{0}^1 t^{\max\{m, n\}-2}\rho([0, t))dt \nonumber \\
&=& (\max\{m, n\}-1)\int_{0}^1 t^{\max\{m, n\}-2}\rho([t, 1))dt.\nonumber
\end{eqnarray}
If $\rho$ is a $[1+\frac{1}{p}(\mu-\nu)]$-Carleson measure on $[0, 1)$, then we see that there is a constant $C_1>0$ such that
$$\rho([t,1))\leq C_1 (1-t)^{1+\frac{1}{p}(\mu-\nu)}$$
holds for all $t\in [0,1)$. It follows that
\begin{eqnarray}\mathcal{I}_{\lambda}[m, n] &\leq & C_1(\max\{m, n\}-1)\int_{0}^1 t^{\max\{m, n\}-2}(1-t)^{1+\frac{1}{p}(\mu-\nu)}dt\nonumber \\&=&C_1 \frac{(\max\{m, n\}-1)\Gamma(\max\{m, n\}-1)\Gamma(2+\frac{1}{p}(\mu-\nu))}{\Gamma(\max\{m, n\}+1+\frac{1}{p}(\mu-\nu))}.\nonumber \end{eqnarray}
By using (\ref{g}) again, we obtain that
$$\frac{(\max\{m, n\}-1)\Gamma(\max\{m, n\}-1)\Gamma(2+\frac{1}{p}(\mu-\nu))}{\Gamma(\max\{m, n\}+1+\frac{1}{p}(\mu-\nu))}
 \asymp \frac{1}{\max\{m, n\}^{1+\frac{1}{p}(\mu-\nu)}}.$$
It follows that (\ref{mun-1}) holds for $m\geq 1, n\geq 2$ or $m\geq 2, n\geq 1$.

Next we consider the case $m=n=1$, we see from the fact $\rho$ is a finite measure on $[0,1)$ that
\begin{equation}
\mathcal{I}_{\lambda}[1, 1]=\int_{0}^1 d\rho(t)=\rho([0, 1))\preceq 1.\nonumber
\end{equation}
Then we get that  (\ref{mun-1}) holds for all $m, n\geq 1$.  Similarly, if $\rho$ is a vanishing $[1+\frac{1}{p}(\mu-\nu)]$-Carleson measure on $[0, 1)$, by minor modifications of above arguments, we can show that (\ref{mun-2}) holds.  The lemma is proved.
\end{proof}

\section{\bf {Proof of Theorem \ref{m-1-2}}}
For $a=\{a_n\}_{n=1}^{\infty} \in l^p, m\geq 1$,  we have
\begin{eqnarray}\lefteqn{m^{\frac{1}{p}[(\alpha-1)+\alpha \mu]}\left|\sum_{n=1}^{\infty}
\frac{n^{\frac{1}{p'}[(\beta-1)-(p'-1)\beta \nu]}}{ [\max\{m^{\alpha}, n^{\beta}\}]^{1+\frac{\mu-\nu}{p}}}a_n\right|}\nonumber  \\&&\leq \sum_{n=1}^{\infty} \left\{[K(m,n)]^{\frac{1}{p}}O_1(m, n)\cdot [K(m,n)]^{\frac{1}{p'}}O_2(m, n)\right\}:=I(m).\nonumber \end{eqnarray}
Where
\begin{equation*}
{K}(m, n)=\frac{1}{[\max\{m^{\alpha}, n^{\beta}\}]^{1+\frac{\mu-\nu}{p}}},
\end{equation*}
\begin{equation*}
{O}_1(m, n)=
\frac{n^{\frac{\beta(1+\nu)}{pp'}-\frac{\beta \nu}{p}}}
{m^{\frac{\alpha(p-1-\mu)}{p^2}}}
\cdot
m^{\frac{1}{p}(\alpha-1)}\cdot |a_n|,
\end{equation*} 
\begin{equation*}
{O}_2(m, n)=
\frac{m^{\frac{\alpha(p-1-\mu)}{p^2}+\frac{\alpha\mu}{p}}}
{n^{\frac{\beta(1+\nu)}{pp'}}}
\cdot
n^{\frac{1}{p'}(\beta-1)}.
\end{equation*}
Applying the H\"{o}lder's inequality on $I(m)$, we get from (\ref{ineq-1}) that
\begin{eqnarray}
I(m) &\leq & \left[\sum_{n=1}^{\infty}{K}(m, n)[{O}_1(m, n)]^p \right]^{\frac{1}{p}}\left[\sum_{n=1}^{\infty}{K}(m, n)[{O}_2(m, n)]^{p'} \right]^{\frac{1}{p'}} \nonumber \\
& =& [{E}(m)]^{\frac{1}{p'}}\left[\sum_{n=1}^{\infty}{K}(m, n)[{O}_1(m, n)]^p \right]^{\frac{1}{p}}. \nonumber
\end{eqnarray}
It follows from  (\ref{ineq-2}) that
\begin{eqnarray}
\lefteqn{\|\widetilde{\mathbf{H}}^{\mu, \nu}_{\alpha, \beta}a\|_p=[\sum_{m=1}^{\infty}I^p(m) ]^{\frac{1}{p}}}  \nonumber \\
& \leq&  \frac{1}{{\beta}^{\frac{1}{p'}}} \left(\frac{p}{1+\mu}+\frac{p}{p-1-\nu}\right)^{\frac{1}{p'}}\left[\sum_{n=1}^{\infty}\sum_{m=1}^{\infty}{K}(m, n)[{O}_1(m, n)]^p \right]^{\frac{1}{p}} \nonumber \\
&=& \frac{1}{{\beta}^{\frac{1}{p'}}} \left(\frac{p}{1+\mu}+\frac{p}{p-1-\nu}\right)^{\frac{1}{p'}} \left[\sum_{n=1}^{\infty}{F}(n) |a_n|^p \right]^{\frac{1}{p}} \nonumber \\
& \leq&  \frac{p}{{\alpha}^{\frac{1}{p}}{\beta}^{\frac{1}{p'}}}\left(\frac{1}{1+\mu}+\frac{1}{p-1-\nu}\right)\|a\|_p. \nonumber
\end{eqnarray}
This means that $\widetilde{\mathbf{H}}^{\mu, \nu}_{\alpha, \beta}$ is bounded on $l^p$ and
\begin{equation}\label{low}\|\widetilde{\mathbf{H}}^{\mu, \nu}_{\alpha, \beta}\|\leq  \frac{p}{{\alpha}^{\frac{1}{p}}{\beta}^{\frac{1}{p'}}}\left(\frac{1}{1+\mu}+\frac{1}{p-1-\nu}\right).\end{equation}
For $\varepsilon>0$, we take $\widetilde{a}=\{\widetilde{a}_n\}_{n=1}^{\infty}$ with $\widetilde{a}_n=\varepsilon^{\frac{1}{p}}n^{-\frac{1+\beta\varepsilon}{p}}$. On the one hand, we have
$$\|\widetilde{a}\|_p^p=\varepsilon \sum_{n=1}^{\infty} n^{-1-\beta\varepsilon}\geq \varepsilon \int_{1}^{\infty} x^{-1-\beta\varepsilon}\, dx=\frac{1}{\beta}.$$
On the other hand, we have
$$\|\widetilde{a}\|_p^p=\varepsilon+ \sum_{n=2}^{\infty}n^{-1-\beta\varepsilon}\leq \varepsilon+\varepsilon\int_{1}^{\infty} x^{-1-\beta\varepsilon}\, dx=\varepsilon+\frac{1}{\beta}.$$
Thus, we obtain that
\begin{equation}\label{est-0}
\|\widetilde{a}\|_p^p=\frac{1}{\beta}(1+o(1)),\,\,   \,\,\varepsilon \rightarrow 0^{+}.
\end{equation}
We write
\begin{eqnarray}\label{shar-1}
\|\widetilde{\mathbf{H}}^{\mu, \nu}_{\alpha, \beta}\widetilde{a}\|_p^p=\varepsilon \sum_{m=1}^{\infty}m^{(\alpha-1)+\alpha \mu }\cdot [J(m)]^p.
\end{eqnarray}
Here $$J(m):=\sum_{n=1}^{\infty}\frac{n^{\frac{1}{p'}[(\beta-1)-(p'-1)\beta \nu]}\cdot n^{-\frac{1+\beta\varepsilon}{p}}}{[\max\{m^{\alpha}, n^{\beta}\}]^{1+\frac{\mu-\nu}{p}}}.$$
In view of  the assumption $0<\beta\leq 1,  -1<\nu<p-1$, we have $1+\beta \nu\geq 0$. Hence we get that
\begin{eqnarray}\frac{1}{p'}[(\beta-1)-(p'-1)\beta \nu]-\frac{1+\beta \varepsilon}{p}=\frac{1}{p'}(\beta-1)-\frac{1+\beta \nu+\beta\varepsilon}{p}<0. \nonumber \end{eqnarray}
Consequently,
\begin{eqnarray}\label{shar-2}
J(m)&\geq& \int_{1}^{\infty} \frac{x^{\frac{1}{p'}[(\beta-1)-(p'-1)\beta \nu]}\cdot x^{-\frac{1+\beta\varepsilon}{p}}}{[\max\{m^{\alpha}, x^{\beta}\}]^{1+\frac{\mu-\nu}{p}}}\,dx \nonumber \\&=&\frac{1}{\beta} \int_{1}^{\infty} \frac{s^{-\frac{1+\nu+\varepsilon}{p}}}{[\max\{m^{\alpha}, s\}]^{1+\frac{\mu-\nu}{p}}}\,ds\nonumber \\ &=&\frac{1}{\beta}m^{-\frac{\alpha}{p}(1+\mu+\varepsilon)} \int_{\frac{1}{m^{\alpha}}}^{\infty} \frac{t^{-\frac{1+\nu+\varepsilon}{p}}}{[\max\{1, t\}]^{1+\frac{\mu-\nu}{p}}}\,dt.
\end{eqnarray}
Also, for $0<\varepsilon< p-1-\nu$, we have
\begin{eqnarray}\label{shar-3}
\lefteqn{\int_{\frac{1}{m^{\alpha}}}^{\infty} \frac{t^{-\frac{1+\nu+\varepsilon}{p}}}{[\max\{1, t\}]^{1+\frac{\mu-\nu}{p}}}\,dt=\int_{0}^{\infty} \frac{t^{-\frac{1+\nu+\varepsilon}{p}}}{[\max\{1, t\}]^{1+\frac{\mu-\nu}{p}}}\,dt -\int_{0}^{\frac{1}{m^{\alpha}}} t^{-\frac{1+\nu+\varepsilon}{p}}\,dt  }\nonumber \\
 &&= p\left(\frac{1}{1+\mu}+\frac{1}{p-1-\nu-\varepsilon}\right)-\frac{p}{p-1-\nu-\varepsilon}m^{-\frac{\alpha(p-1-\nu-\varepsilon)}{p}}
   \nonumber
\\
&&:=L(\varepsilon)-Q(m). \end{eqnarray}
Combining (\ref{shar-1}), (\ref{shar-2}) and (\ref{shar-3}), we get that
\begin{eqnarray}\label{mo-1}
\|\widetilde{\mathbf{H}}^{\mu, \nu}_{\alpha, \beta}\widetilde{a}\|_p^p&\geq &\frac{\varepsilon}{\beta^p}\sum_{m=1}^{\infty} m^{-1-\alpha \varepsilon}\cdot [L(\varepsilon)-Q(m)]^p,
\end{eqnarray}
for $0<\varepsilon< p-1-\nu$.
By using the Bernoulli's inequality(see \cite{M}), we obtain that
\begin{equation}\label{mo-2}
[L(\varepsilon)-Q(m)]^p \geq  [L(\varepsilon)]^p \left [1-\frac{p^2}{L(\varepsilon)(p-1-\nu-\varepsilon)}m^{-\frac{\alpha(p-1-\nu-\varepsilon)}{p}}\right],
\end{equation}
for $0<\varepsilon< p-1-\nu$.
From (\ref{mo-1}) and (\ref{mo-2}), we obtain that
\begin{eqnarray}\label{mo-3}
\|\widetilde{\mathbf{H}}^{\mu, \nu}_{\alpha, \beta}\widetilde{a}\|_p^p&\geq &\frac{\varepsilon}{\beta^p}[L(\varepsilon)]^p\sum_{m=1}^{\infty} m^{-1-\alpha \varepsilon}
\nonumber \\& &-\frac{\varepsilon p^2}{\beta^{p}L(\varepsilon)(p-1-\nu-\varepsilon)}\sum_{m=1}^{\infty} m^{-1-\alpha \varepsilon-\frac{\alpha(p-1-\nu-\varepsilon)}{p}}.
\end{eqnarray}
We note that
\begin{eqnarray}\label{mo-4}
\varepsilon \sum_{m=1}^{\infty}m^{-1-\alpha \varepsilon}=\frac{1}{\alpha}(1+o(1)), \,\, \,\, \varepsilon \rightarrow 0^{+},
\end{eqnarray}
and, for $0<\varepsilon<p-1-\nu$,
\begin{eqnarray}\label{mo-5}
\sum_{m=1}^{\infty}m^{-1-\alpha \varepsilon-\frac{\alpha(p-1-\nu-\varepsilon)}{p}}=O(1), \,\, \varepsilon \rightarrow 0^{+}. 
\end{eqnarray}
It follows from (\ref{mo-3})-(\ref{mo-5}) that
\begin{eqnarray}
\|\widetilde{\mathbf{H}}^{\mu, \nu}_{\alpha, \beta}\widetilde{a}\|_p^p&\geq &\frac{1}{\alpha\beta^p}(1+o(1))\cdot [L(\varepsilon)]^p\cdot [1-\varepsilon O(1)]. \nonumber
\end{eqnarray}
Hence, by (\ref{est-0}), we get that
\begin{eqnarray}
\|\widetilde{\mathbf{H}}^{\mu, \nu}_{\alpha, \beta}\| &\geq &
\frac{\|\widetilde{\mathbf{H}}^{\mu, \nu}_{\alpha, \beta}\widetilde{a}\|_p}{\|\widetilde{a}\|_p}\geq \frac{\frac{1}{\alpha^{\frac{1}{p}}\beta}(1+o(1))\cdot [L(\varepsilon)]\cdot [1-\varepsilon O(1)]^{\frac{1}{p}}}{\frac{1}{\beta^{\frac{1}{p}}}(1+o(1))}. \nonumber
\end{eqnarray}
Take $\varepsilon \rightarrow 0^{+}$, we see that
\begin{equation}\label{big}
\|\widetilde{\mathbf{H}}^{\mu, \nu}_{\alpha, \beta}\|\geq \frac{p}{{\alpha}^{\frac{1}{p}}{\beta}^{\frac{1}{p'}}}\left(\frac{1}{1+\mu}+\frac{1}{p-1-\nu}\right).\end{equation}
Combining (\ref{low}) and (\ref{big}), we see that (\ref{norm}) is true and the proof of Theorem \ref{m-1-2} is finished.

\section{\bf {Proof of Theorem \ref{m-1-1}}}
We first prove the if part.  If $p(\gamma-1)-(\mu-\nu)\geq 0$, that is $\gamma\geq 1+\frac{\mu-\nu}{p}$,  then, for $a=\{a_n\}_{n=1}^{\infty}, \,  m\geq 1$,  it is easy to see that
$$\left|\sum_{n=1}^{\infty}
\frac{n^{\frac{1}{p'}[(\beta-1)-(p'-1)\beta \nu]}}{[\max\{m^{\alpha}, n^{\beta}\}]^{\gamma}}a_n\right|\leq
\sum_{n=1}^{\infty}
\frac{n^{\frac{1}{p'}[(\beta-1)-(p'-1)\beta \nu]}}{[\max\{m^{\alpha}, n^{\beta}\}]^{1+\frac{\mu-\nu}{p}}}|a_n|.$$
Consequently, in view of the boundedness of $\widetilde{\mathbf{H}}^{\mu, \nu}_{\alpha, \beta}$, we conclude that $\mathbf{H}^{\mu, \nu}_{\alpha, \beta, \gamma}$ is bounded on $l^p$ when $p(\gamma-1)-(\mu-\nu)\geq 0$.

Next, we prove the only if part.  We will show that, if $p(\gamma-1)-(\mu-\nu)<0$, then $\mathbf{H}^{\mu, \nu}_{\alpha, \beta, \gamma}$ can not be bounded on $l^p$.
Actually, let $\varepsilon>0$, we still take $\widetilde{a}=\{\widetilde{a}_n\}_{n=1}^{\infty}$ with $\widetilde{a}_n=\varepsilon^{\frac{1}{p}} n^{-\frac{1+\beta\varepsilon}{p}}$.  We have 
\begin{equation}\label{jia-1}\|\widetilde{a}\|_p^p=\frac{1}{\beta}(1+o(1)), \,  \varepsilon \rightarrow 0^{+}.\end{equation}
It follows that
\begin{eqnarray}\label{3-0}\|\mathbf{H}^{\mu, \nu}_{\alpha, \beta, \gamma} \widetilde{a}\|_{p}^p&=&
\sum_{m=1}^{\infty}m^{(\alpha-1)+\alpha \mu}\left[\sum_{n=1}^{\infty}
\frac{n^{\frac{1}{p'}[(\beta-1)-(p'-1)\beta \nu]}\cdot n^{-\frac{1+\beta\varepsilon}{p}}}{[\max\{m^{\alpha}, n^{\beta}\}]^{\gamma}}\right]^p \nonumber \\ & :=& 
 \sum_{m=1}^{\infty}m^{(\alpha-1)+\alpha \mu}\cdot [R(m)]^p. 
\end{eqnarray}
On the other hand, we have, for $m \geq 1$,
\begin{eqnarray}\label{3-1}
R(m)&\geq& \int_{1}^{\infty} \frac{x^{\frac{1}{p'}[(\beta-1)-(p'-1)\beta \nu]}\cdot x^{-\frac{1+\beta\varepsilon}{p}}}{[\max\{m^{\alpha}, x^{\beta}\}]^{\gamma}}\,dx  \\
&=&\frac{1}{\beta} \int_{1}^{\infty} \frac{s^{-\frac{1+\nu+\varepsilon}{p}}}{[\max\{m^{\alpha}, s\}]^{\gamma}}\,ds \nonumber  \\
&=&\frac{1}{\beta}m^{-\frac{\alpha}{p}[p(\gamma-1)+(1+\nu+\varepsilon)]} \int_{\frac{1}{m^{\alpha}}}^{\infty} \frac{t^{-\frac{1+\nu+\varepsilon}{p}}}{[\max\{1, t\}]^{\gamma}}\,dt  \nonumber \\
&\geq &  \frac{1}{\beta}m^{-\frac{\alpha}{p}[p(\gamma-1)+(1+\nu+\varepsilon)]} \int_{1}^{\infty}t^{-\frac{1+\nu+\varepsilon}{p}-\gamma}\,dt .\nonumber
\end{eqnarray}
Since  $p(\gamma-1)-(\mu-\nu)<0$, i.e., $\gamma<1+\frac{\mu-\nu}{p}$, we get that 
\begin{equation}\label{3-2} \int_{1}^{\infty}t^{-\frac{1+\nu+\varepsilon}{p}-\gamma}\,dt  \geq  \int_{1}^{\infty}t^{-\frac{1+\mu+\varepsilon}{p}-1}\,dt=\frac{p}{1+\mu+\varepsilon}.\end{equation}
Consequently, from (\ref{3-0})-(\ref{3-2}), we obtain that 
\begin{eqnarray}\|\mathbf{H}^{\mu, \nu}_{\alpha, \beta, \gamma} \widetilde{a}\|_{p}^p&\geq &\frac{p^p}{[\beta(1+\mu+\varepsilon)]^p}\left[\sum_{m=1}^{\infty} m^{\alpha[p(\gamma-1)+(\mu-\nu)-\varepsilon]-1}\right]. \nonumber
\end{eqnarray}
We suppose that $\mathbf{H}^{\mu, \nu}_{\alpha, \beta, \gamma}: l^p \rightarrow l^p$ is bounded, it follows from (\ref{jia-1}) that  
\begin{eqnarray}\label{c-1}
+\infty&>&\frac{\|\mathbf{H}^{\mu, \nu}_{\alpha, \beta, \gamma}\widetilde{a}\|_p^p}{\|\widetilde{a}\|_p^p}\nonumber\\&\geq&(1+o(1)) \frac{p^p}{[\beta(1+\mu+\varepsilon)]^p}\left[\sum_{m=1}^{\infty} m^{\alpha[p(1-\gamma)+(\mu-\nu)-\varepsilon]-1}\right].
\end{eqnarray}
However, by $p(\gamma-1)-(\mu-\nu)<0$, we know that $p(1-\gamma)+(\mu-\nu)>0$.  Hence, when $\varepsilon<p(1-\gamma)+(\mu-\nu)$, we see from $p(1-\gamma)+(\mu-\nu)-\varepsilon:=\theta>0$ that
$$\sum_{m=1}^{\infty}m^{\alpha[p(1-\gamma)+(\mu-\nu)-\varepsilon]-1} =\sum_{m=1}^{\infty}m^{\theta\alpha-1} =+\infty.$$
Thus we get that (\ref{c-1}) is a contradiction. This proves that $\mathbf{H}^{\mu, \nu}_{\alpha, \beta, \gamma}$ can not be bounded on $l^p$, if $p(\gamma-1)-(\mu-\nu)<0$. Theorem \ref{m-1-1} is proved.

\section{\bf {Proof of Theorem \ref{m-1-3} and \ref{m-1-4}}}

We shall first prove Theorem \ref{m-1-3}.  Firstly, we prove the if part of Theorem \ref{m-1-3}.  By Lemma \ref{lem} and checking the proof of Theorem \ref{m-1-2}, we see that $ \widehat{\mathbf{H}}^{\mu, \nu}_{\gamma, \lambda}$ is bounded on $l^p$, if $d\rho(t)=(1-t)^{\lambda-1}d\lambda(t)$ is a $[1+\frac{1}{p}(\mu-\nu)]$-Carleson measure on $[0, 1)$. The if part of Theorem \ref{m-1-3} is proved.

Secondly, we will show the only if part of Theorem \ref{m-1-3}. In our proof, we need the following well-known estimate, see \cite[Page 54]{Zh}.  Let $0<w<1$. For any $c>0$, we have
\begin{equation}\label{est}\sum_{n=1}^{\infty}n^{c-1}w^{2n}\asymp \frac{1}{(1-w^2)^c}.
\end{equation}
For $0<w<1$. We define $\mathbf{a}=\{\mathbf{a}\}_{n=1}^{\infty}$ as
\begin{equation}\label{re}
\mathbf{a}_n=(1-w^2)^{\frac{1}{p}}w^{\frac{2}{p}(n-1)}, n\in \mathbb{N}.\end{equation}
Then it is easy to  see that $\|\mathbf{a}\|_{p}=1.$  In view of the boundedness of $\widehat{\mathbf{H}}^{\mu, \nu}_{\gamma, \lambda}$, we obtain that
\begin{eqnarray}\label{last}
1 &\succeq& \|\widehat{\mathbf{H}}^{\mu, \nu}_{\gamma, \lambda}\mathbf{a}\|_{p}^p \nonumber \\
&=&\sum_{m=1}^{\infty}m^{\mu}\Bigg|\sum_{n=1}^{m}\mathbf{a}_n n^{-\frac{\nu}{p}}\int_{0}^1t^{m-1}d\rho(t) +\sum_{n=m+1}^{\infty}\mathbf{a}_n n^{-\frac{\nu}{p}}\int_{0}^1t^{n-1}d\rho(t) \Bigg|^p \nonumber \\
&=&(1-w^2)\sum_{m=1}^{\infty}m^{\mu}\Bigg|\sum_{n=1}^{m} w^{\frac{2}{p}(n-1)} n^{-\frac{\nu}{p}}\int_{0}^1t^{m-1}d\rho(t) 
\nonumber \\
&&\quad\quad\quad+\sum_{n=m+1}^{\infty}w^{\frac{2}{p}(n-1)} n^{-\frac{\nu}{p}}\int_{0}^1t^{n-1}d\rho(t) \Bigg|^p.
\nonumber
\end{eqnarray}

(${\bf {I}}$) When $0\leq \nu<p-1$, we see that 
\begin{eqnarray}\label{re-1}
1& \succeq& \|\widehat{\mathbf{H}}^{\mu, \nu}_{\gamma, \lambda}\mathbf{a}\|_{p}^p\geq (1-w^2)\sum_{m=1}^{\infty}m^{\mu}\Bigg|\sum_{n=1}^{m} w^{\frac{2}{p}(n-1)} n^{-\frac{\nu}{p}}\int_{0}^1t^{m-1}d\rho(t) \Bigg|^p \nonumber \\
&\geq & (1-w^2)\sum_{m=1}^{\infty}m^{\mu}\Bigg|\sum_{n=1}^{m} w^{\frac{2}{p}(n-1)} n^{-\frac{\nu}{p}}\int_{w}^1t^{m-1}d\rho(t) \Bigg|^p \nonumber \\
&\geq & (1-w^2)[\rho([w,1))]^p\sum_{m=1}^{\infty}m^{\mu}w^{p(m-1)}\Bigg[\sum_{n=1}^{m} w^{\frac{2}{p}(n-1)} n^{-\frac{\nu}{p}} \Bigg]^p. 
\end{eqnarray}
On the other hand, we note that, for any $m\geq 1$,
\begin{equation*}\label{n-1}\sum_{n=1}^{m}w^{\frac{2}{p}(n-1)}n^{-\frac{\nu}{p}}\geq m\cdot  w^{\frac{2}{p}(m-1)}m^{-\frac{\nu}{p}}=w^{\frac{2}{p}(m-1)}m^{1-\frac{\nu}{p}}.\end{equation*}
Then we get that
\begin{equation}
\sum_{m=1}^{\infty}m^{\mu}w^{p(m-1)}\Bigg[\sum_{n=1}^{m} w^{\frac{2}{p}(n-1)} n^{-\frac{\nu}{p}} \Bigg]^p\geq \sum_{m=1}^{\infty} m^{\mu+p(1-\frac{\nu}{p})}w^{(p+2)(m-1)}.
\nonumber
\end{equation}
It follows from (\ref{re-1}) that
\begin{equation}\label{end}
1\succeq(1-w^2)[\rho([w, 1))]^{p} \sum_{m=1}^{\infty} m^{\mu+p(1-\frac{\nu}{p})}w^{(p+2)(m-1)}. \nonumber
\end{equation}
Then we conclude from (\ref{est}) that
$$(1-w^2)[\rho([w, 1))]^{p}\frac{1}{(1-w^2)^{\mu+p(1-\frac{\nu}{p})+1}}\preceq 1.$$
This implies that
$$\rho([w, 1))\preceq (1-w^2)^{1+\frac{1}{p}(\mu-\nu)},\,\, {\text {for all}}\,\, w\in (0,1).$$

(${\bf {II}}$) When $-1<\nu<0$, we see that 
\begin{eqnarray}\label{re-2}
1& \succeq& \|\widehat{\mathbf{H}}^{\mu, \nu}_{\gamma, \lambda}\mathbf{a}\|_{p}^p\geq (1-w^2)\sum_{m=1}^{\infty}m^{\mu}\Bigg|\sum_{n=m+1}^{\infty} w^{\frac{2}{p}(n-1)} n^{-\frac{\nu}{p}}\int_{0}^1t^{n-1}d\rho(t) \Bigg|^p \nonumber \\
&\geq & (1-w^2)\sum_{m=1}^{\infty}m^{\mu}\Bigg|\sum_{n=m+1}^{\infty} w^{\frac{2}{p}(n-1)} n^{-\frac{\nu}{p}}\int_{w}^1t^{n-1}d\rho(t) \Bigg|^p \nonumber \\
&\geq & (1-w^2)[\rho([w,1))]^p\sum_{m=1}^{\infty}m^{\mu}\Bigg[\sum_{n=m+1}^{\infty} w^{(\frac{2}{p}+1)(n-1)} n^{-\frac{\nu}{p}} \Bigg]^p. 
\end{eqnarray}
Meanwhile, we note that, for any $m\geq 1$,
\begin{eqnarray}
\lefteqn{\sum_{n=m+1}^{\infty} w^{(\frac{2}{p}+1)(n-1)} n^{-\frac{\nu}{p}} \geq\sum_{n=m+1}^{\infty} w^{(\frac{2}{p}+1)(n-1)} m^{-\frac{\nu}{p}} } \nonumber\\
&&=m^{-\frac{\nu}{p}} \frac{w^{(\frac{2}{p}+1)m}}{1-w^{\frac{2}{p}+1}} \succeq m^{-\frac{\nu}{p}} \frac{w^{(\frac{2}{p}+1)m}}{1-w^{2}}.\nonumber 
\end{eqnarray}
Then we get that
\begin{equation}
\sum_{m=1}^{\infty}m^{\mu}\Bigg[\sum_{n=m+1}^{\infty} w^{(\frac{2}{p}+1)(n-1)} n^{-\frac{\nu}{p}} \Bigg]^p \succeq \frac{1}{(1-w^{2})^p}\sum_{m=1}^{\infty}m^{\mu-\nu}w^{(p+2)m}. \nonumber 
\end{equation}
It follows from (\ref{re-2}) that
\begin{equation}
1\succeq\frac{1-w^2}{(1-w^{2})^p}[\rho([w, 1))]^{p} \sum_{m=1}^{\infty} m^{\mu-\nu}w^{(p+2)m}. \nonumber
\end{equation}
Then, from again (\ref{est}), we see that
$$(1-w^2)[\rho([w, 1))]^{p}\frac{1}{(1-w^2)^{\mu-\nu+p+1}}\preceq 1.$$
This also implies that
$$\rho([w, 1))\preceq (1-w^2)^{1+\frac{1}{p}(\mu-\nu)},\,\, {\text {for all}}\,\, w\in (0,1).$$
Combining (${\bf {I}}$) and (${\bf {II}}$), we see that $\rho$ is a $[1+\frac{1}{p}(\mu-\nu)]$-Carleson measure on $[0, 1)$ and the only if part of  \ref{m-1-3} is proved.  
Now, the proof of Theorem \ref{m-1-3} is finished.

We next prove Theorem \ref{m-1-4}.  We first show the if part.  We assume that $\rho$ is a vanishing $[1+\frac{1}{p}(\mu-\nu)]$-Carleson measure on $[0, 1)$. Let $\mathfrak{M}\in \mathbb{N}$, we define the operator $\mathbf{H}^{[\mathfrak{M}]}$ as, for $a=\{a_n\}_{n=1}^{\infty}$, $$\mathbf{H}^{[\mathfrak{M}]}(a)(m):=m^{\frac{\mu}{p}}\sum_{n=1}^{\infty}n^{-\frac{\nu}{p}}\mathcal{I}_{\lambda}[m, n]a_n,$$
when $m\leq \mathfrak{M}$, and $\mathbf{H}^{[\mathfrak{M}]}(a)(m):=0,$ when $m\geq \mathfrak{M}+1$.
Then we see that $\mathbf{H}^{[\mathfrak{M}]}$ is a finite rank operator and hence it is compact on $l^{p}$.
By Lemma \ref{lem}, we know that, for any $\epsilon>0$, there is an ${\mathbf{M}}\in \mathbb{N}$ such that
$$\mathcal{I}_{\lambda}[m, n]\preceq \frac{\epsilon}{[\max\{m, n\}]^{1+\frac{1}{p}(\mu-\nu)}}$$
holds for all $n\geq 1, m>{\mathbf{M}}$.
Then, we see from
\begin{eqnarray}\|(\widehat{\mathbf{H}}^{\mu, \nu}_{\gamma, \lambda}-\mathbf{H}^{[\mathfrak{M}]})a\|_{p}^p=\sum_{m=\mathfrak{M}+1}^{\infty}m^{{\mu}}\left|\sum_{n=1}^{\infty}n^{-\frac{\nu}{p}}\mathcal{I}_{\lambda}[m, n]a_n\right|^p,\nonumber
\end{eqnarray}
that,
\begin{eqnarray}\|(\widehat{\mathbf{H}}^{\mu, \nu}_{\gamma, \lambda}-\mathbf{H}^{[\mathfrak{M}]})a\|_{p}^p\preceq
 {\epsilon}^p \sum_{m=\mathfrak{M}+1}^{\infty}m^{{\mu}}\left|\sum_{n=1}^{\infty}\frac{a_n}{[\max\{m, n\}]^{1+\frac{1}{p}(\mu-\nu)}}\right|^p,\nonumber\end{eqnarray}
when $\mathfrak{M}>{\mathbf{M}}.$
Consequently, by checking the proof of Theorem \ref{m-1-2},  we see that, for any $\epsilon>0$, it holds that
\begin{eqnarray}\|(\widehat{\mathbf{H}}^{\mu, \nu}_{\gamma, \lambda}-\mathbf{H}^{[\mathfrak{M}]})a\|_{p}\preceq
 {\epsilon}  \|a\|_p, \nonumber\end{eqnarray}
for all $a\in l^p$ when $\mathfrak{M}>{\mathbf{M}}$. It follows that $\widehat{\mathbf{H}}^{\mu, \nu}_{\gamma, \lambda}$ is compact on $l^p$. This proves the if part of Theorem \ref{m-1-4} .

Finally, we prove the only if part.  For $0<w<1$.  We take $\mathbf{{a}}=\{\mathbf{{a}}_n\}_{n=1}^{\infty}$ as in (\ref{re}). 
It is easy to check that  $\{\mathbf{{a}}_n\}_{n=1}^{\infty}$ is convergent weakly to $0$ on $l^{p}$ as $w\rightarrow 1^{-}$. Since $\widehat{\mathbf{H}}^{\mu, \nu}_{\gamma, \lambda}$ is compact on $l^p$, we get that
\begin{equation}\label{com-0}\lim_{w\rightarrow {1^{-}}} \|\widehat{\mathbf{H}}^{\mu, \nu}_{\gamma, \lambda}\mathbf{{a}}\|_{p}=0.\end{equation}
On the other hand, by checking the arguments of the proof of Theorem \ref{m-1-3}, we have
\begin{eqnarray}
\|\widehat{\mathbf{H}}^{\mu, \nu}_{\gamma, \lambda}\mathbf{{a}}\|_{p}^p
&\succeq&[\rho([w,1))]^p\cdot \frac{1}{(1-w^2)^{\mu-\nu+p}}.\nonumber
\end{eqnarray}
This yields that
\begin{eqnarray}
\rho([w,1))\preceq \|\widehat{\mathbf{H}}^{\mu, \nu}_{\gamma, \lambda}{\mathbf{{a}}}\|_{p} (1-w^2)^{1+\frac{1}{p}(\mu-\nu)}.\nonumber
\end{eqnarray}
It follows from (\ref{com-0}) that $\rho$ is a vanishing $[1+\frac{1}{p}(\mu-\nu)$-Carleson measure on $[0, 1)$. This proves the only if part of Theorem \ref{m-1-4}  and the proof of Theorem \ref{m-1-4} is completed.

\section{\bf {Final Remarks}}
\begin{remark} We first point out that the assumptions $-1<\mu, \nu<p-1$ in Theorem \ref{m-1-1} and \ref{m-1-2} are both necessary.  We consider the case $\alpha=\beta=\gamma=1, \mu=\nu:=\delta$. That is to say, we will consider the operator
$$\mathbf{H}^{\delta, \delta}_{1, 1, 1}(a)(m)=m^{\frac{\delta}{p}}\sum_{n=1}^{\infty}
\frac{n^{-\frac{\delta}{p}}a_n}{\max\{m, n\}} ,\, a=\{a_n\}_{n=1}^{\infty}, \,  m\geq 1. $$

We will use $\overline{\mathbf{H}}_{\delta}$ to denote $\mathbf{H}^{\delta, \delta}_{1, 1, 1}$. We shall show that

\begin{proposition}$\overline{\mathbf{H}}_{\delta}$ is not bounded on $l^p$, if $\delta\leq -1$, or $\delta\geq p-1$.\end{proposition}
\begin{proof}
For $\varepsilon>0$, we take $\bar{a}=\{\bar{a}_n\}_{n=1}^{\infty}$ with $\bar{a}_n=\varepsilon^{\frac{1}{p}}n^{-\frac{1+\varepsilon}{p}}$.  We see that
$$\|\bar{a}\|_p^p=\varepsilon+ \varepsilon\sum_{n=2}^{\infty}n^{-1-\varepsilon}\leq \varepsilon+1,$$
and
$$\|\overline{\mathbf{H}}_{\delta}\bar{a}\|_p^p=\sum_{m=1}^{\infty}m^{\delta}\left[\sum_{n=1}^{\infty}
\frac{n^{-\frac{\delta+1+\varepsilon}{p}}}{\max\{m, n\}}\right]^p.$$

(${\bf {I}}$)  If $\delta< -1$,  when $\varepsilon<-(\delta+1)$, we see from $\delta+1+\varepsilon<0$ that, for any fixed $m \geq 1$, it holds that
$$\sum_{n=1}^{\infty}
\frac{n^{-\frac{\delta+1+\varepsilon}{p}}}{\max\{m, n\}}\geq \sum_{n=m}^{\infty}
\frac{n^{-\frac{\delta+1+\varepsilon}{p}}}{n}\geq \sum_{n=m}^{\infty}
n^{-1-\frac{\delta+1+\varepsilon}{p}}=+\infty. $$
This means that $\overline{\mathbf{H}}_{\delta}$ is not bounded on $l^p$ in this case.

(${\bf {II}}$)  If $\delta=-1$ or $\delta\geq p-1$, for all $m\geq 1$, we have
\begin{eqnarray}\sum_{n=1}^{\infty}
\lefteqn{\frac{ n^{-\frac{\delta+1+\varepsilon}{p}}}{\max\{m, n\}} \geq \int_{1}^{\infty}
\frac{x^{-\frac{\delta+1+\varepsilon}{p}}}{\max\{m, x\}}\,dx} \nonumber  \\ &=&m^{-\frac{\delta+1+\varepsilon}{p}} \int_{\frac{1}{m}}^{\infty}
\frac{t^{-\frac{\delta+1+\varepsilon}{p}}}{\max\{1, t\}}\,dt  \nonumber 
\\&\geq &m^{-\frac{\delta+1+\varepsilon}{p}}\int_{1}^{\infty}
t^{-1-\frac{\delta+1+\varepsilon}{p}}\,dt=\frac{p}{1+\delta+\varepsilon}m^{-\frac{\delta+1+\varepsilon}{p}}. \nonumber \end{eqnarray}
Consequently,
\begin{eqnarray}\|\overline{\mathbf{H}}_{\delta}\bar{a}\|_p^p\geq \left(\frac{p}{1+\delta+\varepsilon}\right)^p\cdot \sum_{m=1}^{\infty}\frac{1}{m^{1+\varepsilon}}. \nonumber \end{eqnarray}
On the other hand, we have 
\begin{eqnarray}
\sum_{m=1}^{\infty}\frac{1}{m^{1+\varepsilon}}=\frac{1}{\varepsilon}(1+o(1)), \, \varepsilon \rightarrow 0^{+}.\nonumber
\end{eqnarray}
Therefore, we get that
\begin{eqnarray}\|\overline{\mathbf{H}}_{\delta}\bar{a}\|_p^p\geq  \left(\frac{p}{1+\delta+\varepsilon}\right)^p \frac{1}{\varepsilon}(1+o(1)).\nonumber \end{eqnarray}
Taking $\varepsilon \rightarrow 0^{+}$, we obtain that $\|\overline{\mathbf{H}}_{\delta}\bar{a}\|_p^p \rightarrow +\infty$.  This implies that $\overline{\mathbf{H}}_{\delta}$ is not bounded on $l^p$ when $\delta=-1$ or $\delta\geq p-1$.
The proposition is proved.
\end{proof}
\end{remark}

\begin{remark}
When $\gamma=1$, from Theorem \ref{m-1-3} and \ref{m-1-4}, we have
\begin{corollary}
Let $p>1$ and $-1<\mu, \nu<p-1$.  Let $\lambda$ be a positive finite Borel measure on $[0, 1)$ and  $\check{\mathbf{H}}^{\mu, \nu}_{\lambda}$ be defined as 
$$\check{\mathbf{H}}^{\mu, \nu}_{\lambda}(a)(m):=\sum_{n=1}^{\infty} m^{\frac{\mu}{p}}n^{-\frac{\nu}{p}}\check{\mathcal{I}}_\lambda[m,n]a_n,  \, a=\{a_n\}_{n=1}^{\infty}, \,  m\geq 1.$$
Here \begin{equation}\check{\mathcal{I}}_\lambda[m, n]=\int_{[0, 1)}t^{\max\{m, n\}-1}d\lambda(t), \,m, n\geq 1. \nonumber \end{equation}
Then $\check{\mathbf{H}}^{\mu, \nu}_{\lambda}$ is bounded (compact) on $l^p$ if and only if
$\lambda$ is a (vanishing) $[1+\frac{1}{p}(\mu-\nu)]$-Carleson measure on $[0, 1)$, respectively. 
\end{corollary}
\end{remark}

\begin{remark}
We finally consider the Hardy-Littlewood-P\'olya-type operator acting on the analytic function spaces in the unit disk $\mathbb{D}$.  Let $\mathcal{A}(\mathbb{D})$ be the class of all analytic functions in the unit disk $\mathbb{D}$ of the complex plane. These years, for a function $f(z)=\sum_{n=0}^{\infty}a_n z^n \in \mathcal{A}(\mathbb{D})$, the following Hilbert operator $\mathcal{H}$, acting on the Taylor coefficients of $f$, and its variants and generalizations have been extensively studied, see \cite{BK},  \cite{D-1, D, DS, DJV}, \cite{GGPS, GP, GM-2}, \cite{Ka}, \cite{LMW}, \cite{LMN}, \cite{PR}. 
\begin{equation}\label{hhh-1}\mathcal{H}(f)(z):=\sum_{m=0}^{\infty}\Bigg[\sum_{n=0}^{\infty}\frac{a_n}{m+n+1}\Bigg]z^m.  \nonumber
\end{equation}
For $\gamma>0$, $f=\sum_{n=0}^{\infty}a^n z^n\in \mathcal{A}(\mathbb{D})$, we similarly define the Hardy-Littlewood-P\'olya-type operator ${\mathbb{H}}_{\gamma}$ as
\begin{equation*} {\mathbb{H}}_{\gamma}(f)(z):=\sum_{m=0}^{\infty} \left(\sum_{n=0}^{\infty} \frac{a_n}{[\max\{m+1, n+1\}]^{\gamma}}\right)z^m.
\end{equation*}
We will investigate the boundedness of ${\mathbb{H}}_{\gamma}$ acting on certain spaces of analytic functions in $\mathbb{D}$.

Let $q$ be a positive number and  $X_q$ be a Banach space of analytic functions in $\mathbb{D}$. For any $f\in X_q$, we assume that the norm $\|f\|_{X_{q}}$  of $f$ is determined by  $f$, $q$ and other finite parameters $\beta_1, \beta_2, \cdots, \beta_k$.  Here $k$ is a non-negative integer and $k=0$ means that there is no parameter.

We denote by $\mathcal{P}(\mathbb{D})$ the class of all functions $f=\sum_{n=0}^{\infty}a_n z^n\in \mathcal{H}(\mathbb{D})$ with $\{a_n\}_{n=0}^{\infty}$ is a decreasing sequence of non-negative real numbers. We say $X_q$ have the {\em sequence-like property}, if, for a function $f\in \mathcal{P}(\mathbb{D})$,  there is a constant $\mathbb{I}_X=\mathbb{I}_{X}(q, \beta_1, \beta_2, \cdots, \beta_k)$ with $\mathbb{I}_X>-1$ such that  $f \in  X_q$  if and only if
$$\sum_{n=0}^{\infty} (n+1)^{\mathbb{I}_X}a_n^{q}<+\infty.$$

We point out that many classical spaces of analytic functions in $\mathbb{D}$ have the sequence-like property.   Let $f=\sum_{n=0}^{\infty}a_n z^n\in \mathcal{P}(\mathbb{D})$.
For example, 

(\dag) the Hardy space $H^q(\mathbb{D}), 1<q<\infty$, we know that, see \cite[page 127]{Pa},  $f \in H^q(\mathbb{D})$ if and only if $$\sum_{n=0}^{\infty} (n+1)^{q-2}a_n^{q}<+\infty.$$

(\dag\dag) For $1<q<\infty$,  let $-2<\alpha\leq q-1$.  It holds that, see \cite[Lemma 4]{GM-2}, $f \in \mathcal{D}^q_{\alpha}(\mathbb{D})$ if and only if
 $$\sum_{n=0}^{\infty} (n+1)^{2q-3-\alpha}a_n^{q}<+\infty.$$
 Here $\mathcal{D}^q_{\alpha}(\mathbb{D})$ is the Dirichlet-type space, defined as
 \begin{eqnarray}\lefteqn{\mathcal{D}^q_{\alpha}(\mathbb{D})=\Bigg\{f\in \mathcal{H}({\mathbb{D}}): \|f\|_{\mathcal{D}^q_{\alpha}}=|f(0)|}
\nonumber\\ &&\quad\quad\quad\quad+\Bigg[(\alpha+1)\int_{\mathbb{D}}|f'(z)|^q(1-|z|^2)^{\alpha}dA(z)\Bigg]^{\frac{1}{q}}<+\infty\Bigg\}.\nonumber \end{eqnarray}

 (\dag\dag\dag) For $1<q<\infty$, let $-1<\alpha<q+2$. It holds that, see \cite[Proposition 1]{GM-2},  $f \in \mathcal{A}^q_{\alpha}(\mathbb{D})$ if and only if
 $$\sum_{n=0}^{\infty} (n+1)^{q-3-\alpha}a_n^{q}<+\infty.$$
 Here $\mathcal{A}^q_{\alpha}(\mathbb{D})$ is the Bergman space, defined as
 $$\mathcal{A}^q_{\alpha}(\mathbb{D})=\left\{f\in \mathcal{H}({\mathbb{D}}): \|f\|_{\mathcal{A}^q_{\alpha}}^q=(\alpha+1)\int_{\mathbb{D}}|f(z)|^q(1-|z|^2)^{\alpha}dA(z)<+\infty\right\}.$$

 We obtain that
\begin{proposition}\label{rem}
Let $\gamma$, $q$ be two positive numbers. Let $X_q$ be a Banach space of analytic functions in $\mathbb{D}$ which has the sequence-like property and  ${\mathbb{H}}_{\gamma}$ be as above. Then the necessary condition of ${\mathbb{H}}_{\gamma}: X_q\rightarrow X_q$ is bounded  is $\gamma\geq 1$.
\end{proposition}

\begin{proof}
We will prove that,  ${\mathbb{H}}_{\gamma}: X_q\rightarrow X_q$  can not be bounded, if $0<\gamma<1$. Let $\varepsilon>0$ and set $\widetilde{f}_{\varepsilon}=\sum_{n=0}^{\infty}\widetilde{a}_n z^n$ with
$\widetilde{a}_n=(\frac{\varepsilon}{1+\varepsilon})^{\frac{1}{q}}(n+1)^{-\frac{\mathbb{I}_X+1+\varepsilon}{q}}.$ It is easy to see that $\{\widetilde{a}_n\}_{n=0}^{\infty}$ is a decreasing sequence and $\sum_{n=0}^{\infty}(n+1)^{\mathbb{I}_{X}}\widetilde{a}_n^{q}<\infty.$ Hence $\widetilde{f}_\varepsilon \in X_q$. We set
$$b_m=\sum_{n=0}^{\infty}\frac{\widetilde{a}_n}{[\max\{m+1, n+1\}]^{\gamma}}, \, m\geq 0.$$
We suppose that ${\mathbb{H}}_{\gamma}: X_q\rightarrow X_q$  is bounded. Then, by the fact that $\{b_n\}_{n=0}^{\infty}$ is a decreasing sequence, we see that $g(z)=\sum_{n=0}^{\infty}b_n z^n \in X_q$ and hence
$$\sum_{m=0}^{\infty} (m+1)^{\mathbb{I}_X} b_n^q <+\infty.$$
That is 
\begin{eqnarray}\label{e-1}
+ \infty &>& \sum_{m=0}^{\infty} (m+1)^{\mathbb{I}_X} \left[\sum_{n=0}^{\infty}\frac{\widetilde{a}_n}{[\max\{m+1, n+1\}]^{\gamma}}\right]^q  \\ &=& 
 \frac{\varepsilon}{1+\varepsilon} \sum_{m=0}^{\infty} (m+1)^{\mathbb{I}_X} \left[\sum_{n=0}^{\infty}\frac{(n+1)^{-\frac{\mathbb{I}_X+1+\varepsilon}{q}}}{[\max\{m+1, n+1\}]^{\gamma}}\right]^q.\nonumber
\end{eqnarray}
On the other hand, for any $m\geq 0$, we have 
\begin{eqnarray}\label{e-2}
\lefteqn{\sum_{n=0}^{\infty}\frac{(n+1)^{-\frac{\mathbb{I}_X+1+\varepsilon}{q}}}{[\max\{m+1, n+1\}]^{\gamma}}=\sum_{n=1}^{\infty}\frac{n^{-\frac{\mathbb{I}_X+1+\varepsilon}{q}}}{[\max\{m+1, n\}]^{\gamma}}}\\
&\geq& \int_{1}^{\infty} \frac{x^{-\frac{\mathbb{I}_X+1+\varepsilon}{q}}}{[\max\{m+1,x\}]^{\gamma}}\,dx \nonumber \\ & =&
(m+1)^{(1-\lambda)-\frac{\mathbb{I}_X+1+\varepsilon}{q}} \cdot \int_{\frac{1}{m+1}}^{\infty} \frac{y^{-\frac{\mathbb{I}_X+1+\varepsilon}{q} }}{[\max\{1, y\}]^{\gamma}}\,dy \nonumber \\
&\geq& (m+1)^{(1-\gamma)-\frac{\mathbb{I}_X+1+\varepsilon}{q}}\int_{1}^{\infty} y^{-\gamma-\frac{\mathbb{I}_X+1+\varepsilon}{q} }\,dy \nonumber \\
&:=& (m+1)^{(1-\gamma)-\frac{\mathbb{I}_X+1+\varepsilon}{q}}E(\varepsilon).\nonumber 
\end{eqnarray}
Combining (\ref{e-1}) and (\ref{e-2}), we get that
\begin{eqnarray}\label{e-3}
+\infty &>& \frac{\varepsilon}{1+\varepsilon} [E(\varepsilon)]^q \left[\sum_{m=0}^{\infty}(m+1)^{q(1-\gamma)-1-\varepsilon}\right].\end{eqnarray}
But, if $\gamma< 1$,  when $\varepsilon<q(1-\gamma)$, we have $$\sum_{m=0}^{\infty}(m+1)^{q(1-\gamma)-1-\varepsilon}=+\infty.$$Thus (\ref{e-3}) is a contradiction since $E(\varepsilon)>0$. This means that ${\mathbb{H}}_{\gamma}: X_q\rightarrow X_q$  can not be bounded when $\gamma< 1$.  The proposition is proved.
\end{proof}
\end{remark}

For $\gamma>0$, let $\lambda$ be a positive Borel measure in $[0,1)$, we define the operator 
\begin{equation}\label{hhh-2}\mathbb{H}_{\gamma, \lambda}(f)(z):=\sum_{m=0}^{\infty}\Bigg[\sum_{n=0}^{\infty}\mathbf{I}_{\gamma, \lambda}[m, n]a_n\Bigg]z^m, \,f=\sum_{n=0}^{\infty}a_n z^n \in \mathcal{A}(\mathbb{D}). \nonumber
\end{equation}
Here \begin{equation}\label{las}\mathbf{I}_{\gamma, \lambda}[m, n]=\int_{[0, 1)}t^{\max\{m, n\}}(1-t)^{\gamma-1}d\lambda(t), \,m, n\geq 0. \end{equation}
When $\lambda$ in (\ref{las}) is the Lesbegue measure in $[0, 1)$, we see that 
$$\mathbf{I}_{\gamma, \lambda}[m, n] \asymp \frac{1}{[\max\{m+1, n+1\}]^{\gamma}}, m, n\geq 0.$$

It is interesting to study
\begin{question}
Characterize the measures $\lambda$ such that $\mathbb{H}_{\gamma, \lambda}$ is bounded (compact) from one analytic function space $X$ to another one $Y$. 
\end{question}

\end{document}